\documentclass[11pt]{article}
\usepackage[margin=1in]{geometry}
\usepackage{amsmath,amssymb,amsthm,mathtools}
\usepackage{booktabs}
\usepackage{graphicx}
\usepackage{tikz}
\usepackage{microtype}
\usepackage{xcolor}
\usepackage{hyperref}
\usepackage{enumitem}
\hypersetup{
  colorlinks=true,
  linkcolor=blue!55!black,
  citecolor=blue!55!black,
  urlcolor=blue!55!black,
  pdfauthor={Yi Du},
  pdftitle={When a Relaxed PEP Is Exact: The Sharp Queried-Gradient Rate of Nesterov's Fast Gradient Method}
}

\newtheorem{theorem}{Theorem}[section]
\newtheorem{lemma}[theorem]{Lemma}
\newtheorem{proposition}[theorem]{Proposition}

\theoremstyle{definition}

\newtheorem{remark}[theorem]{Remark}

\newcommand{\R}{\mathbb{R}}
\newcommand{\F}{\mathcal{F}}
\newcommand{\conv}{\operatorname{conv}}
\newcommand{\dist}{\operatorname{dist}}
\newcommand{\Proj}{\operatorname{Proj}}
\newcommand{\one}{\boldsymbol{1}}
\newcommand{\ip}[2]{\left\langle #1,#2\right\rangle}
\newcommand{\norm}[1]{\left\lVert #1\right\rVert}

\title{When a Relaxed PEP Is Exact:\\
The Sharp Queried-Gradient Rate of Nesterov's Fast Gradient Method}
\author{Yi Du}
\date{Revised 25 August 2026}

\begin{document}
\maketitle

\begin{abstract}
We determine the exact worst-case value, at every horizon $N\geq7$, of the smallest queried gradient norm generated by Nesterov's fast gradient method on smooth convex functions.  Let $t_0=1$ and $t_{k+1}=(1+\sqrt{1+4t_k^2})/2$, and let $x_0,\ldots,x_N$ denote the points at which the method evaluates gradients.  For every such $N$ and every dimension $d\geq N-4$, we prove
\[
 \sup_{\substack{f\in\F_{0,L}(\R^d),\ x_\star\in\arg\min f\\
                  \norm{x_0-x_\star}\leq R}}
 \min_{0\leq k\leq N}\norm{\nabla f(x_k)}^2
 =\frac{L^2R^2}{\sum_{k=0}^N t_k^2}.
\]
The relaxed-PEP upper bound is due to Kim and Fessler, who also reported tight numerical solutions of the exact-interpolation PEP at selected horizons.  What remained missing was an analytic matching family valid uniformly over the horizon.  For every $N\geq7$, we construct such a family using an FGM-specific spherical polytope $K_N$ and the standard projection-envelope function
\[
 f_N(x)=\max_{g\in K_N}\left\{\ip{x}{g}-\frac12\norm{g}^2\right\},
 \qquad \nabla f_N(x)=\Proj_{K_N}(x).
\]
Every queried gradient has the same norm, and the vertices of $K_N$ are generated from a three-dimensional seed by a one-dimensional spherical cone lift.  The lift preserves all projection inequalities and raises the adversary dimension by one at each horizon.  The projection/Moreau-envelope template itself is classical; the new ingredients are the FGM-specific algebraic seed, the proof that it attains the relaxed bound, and the common-latitude lift that propagates this exactness to every $N\geq7$.  We state precise hypotheses for that propagation and do not claim that every rank-one relaxed PEP admits such a seed.
\end{abstract}


\section{Introduction}

Performance estimation problems (PEPs) turn worst-case analysis of first-order methods into finite-dimensional semidefinite programs.  Drori and Teboulle~\cite{DroriTeboulle2014} introduced the approach for smooth convex minimization.  Taylor, Hendrickx, and Glineur~\cite{TaylorHendrickxGlineur2017} subsequently proved necessary and sufficient smooth strongly convex interpolation conditions, making the corresponding finite-horizon SDP exact rather than merely a relaxation.

One of the first numerical surprises produced by this framework concerns the gradient norms of Nesterov's fast gradient method (FGM).  Although FGM was designed for objective decrease, numerical solutions of the exact-interpolation PEP suggested an $O(N^{-3/2})$ rate for the best gradient norm among its queried iterates~\cite[Section~4.3]{TaylorHendrickxGlineur2017}.  Kim and Fessler~\cite[Section~5.2]{KimFessler2018} later proved the chain of analytical bounds
\begin{equation}\label{eq:KF-upper-intro}
 \min_{0\leq k\leq N+1}\norm{\nabla f(y_k)}
 \leq \min_{0\leq k\leq N}\norm{\nabla f(x_k)}
 \leq \frac{LR}{\sqrt{S_N}}
 \leq \frac{2\sqrt{3}\,LR}{\sqrt{(N+1)(N^2+6N+12)}},
 \qquad S_N:=\sum_{k=0}^N t_k^2.
\end{equation}
Their proof uses a deliberately reduced set of interpolation inequalities.  Thus, it establishes an upper bound but does not imply that the bound can be attained by a smooth convex function.  Kim and Fessler also reported dimension-unrestricted tight numerical PEP values for the same queried-gradient criterion~\cite[Section~6.2 and Table~3]{KimFessler2018}.  For the tested moderate horizons, the FGM column nearly coincides with the reciprocal of~\eqref{eq:KF-upper-intro} to the displayed precision.  The gap left open was therefore analytical rather than numerical: no global matching function or all-horizon exactness proof was supplied.

PEPit~\cite{GoujaudEtAl2024} has made such finite-horizon experiments reproducible, but a numerical optimum and a floating-point dual certificate do not by themselves give an all-horizon theorem.  Recent work illustrates both sides of that distinction.  Rotaru, Glineur, and Patrinos~\cite{RotaruGlineurPatrinos2026} obtained exact best-gradient rates for gradient descent over the full constant-stepsize range; their algorithmic setting is different from FGM.  Thomsen et al.~\cite{ThomsenEtAl2026} optimize the sparsity and structure of PEP certificates, emphasizing that certificate simplification is a second mathematical problem.  For Nesterov acceleration itself, Jang and Ryu~\cite{JangRyu2026} recently established point convergence, a qualitative infinite-horizon statement rather than the finite-horizon minimax constant studied here.

The lower-bound function family used below must be distinguished from the new rate result.  If $K=aB_2$ is a Euclidean ball, then
\[
 \phi_K(x)=\begin{cases}
 \tfrac12\norm{x}^2,&\norm{x}\leq a,\\
 a\norm{x}-\tfrac12a^2,&\norm{x}\geq a,
 \end{cases}
\]
the radial Huber function that already appears in classical PEP worst-case constructions, including Drori and Teboulle~\cite{DroriTeboulle2014}.  Drori and Taylor~\cite[Appendix~A]{DroriTaylor2022} subsequently used a max-over-simplex smooth extension in oracle-complexity lower bounds; their specific non-strongly-convex lower-bound instance can be rewritten, after translation and scaling, as a polyhedral member of the same max-quadratic family.  Most directly, Ma and Chen~\cite{MaChen2026}, posted on arXiv on 11 August 2026, use $K=\conv\{0,g_0,\ldots,g_m\}$, the Moreau envelope of its support function, the identity $\nabla\phi_K=\Proj_K$, and normal-cone tests to construct lower bounds for stepsize-based gradient descent.  We therefore make no novelty claim for the projection-envelope template, its normal-cone verification, or the use of new orthogonal coordinates in a hard instance.

Dimension growth and horizon recursion also have substantial precedents.  Zero-chain and resisting-oracle lower bounds enlarge the hidden orthogonal subspace as information is revealed~\cite[Section~2.2]{DroriTaylor2022}, while Altschuler and Parrilo~\cite{AltschulerParrilo2025} recursively glue constraints from different trajectory segments on the upper-bound/certificate side.  These constructions do not, to our knowledge, give the specific primal transformation used here: place all vertices of an attained projection-polytope adversary at one common latitude, add one apex, and preserve every pre-existing projection inequality while extending horizon $N$ to $N+1$.  Proposition~\ref{prop:conditional-propagation} formulates only this narrower statement.

This paper closes that analytical gap from horizon seven onward.  To the best of our knowledge, the contribution is the first proof that the Kim--Fessler queried-gradient bound is attained for every $N\geq7$, together with a globally defined matching function at each horizon.  It is not the first numerical PEP evidence for this equality and not the first use of a projection-envelope adversary.  The construction was guided by three features of high-accuracy primal PEP solutions: the queried gradient norms nearly equalize; the numerical gradient Gram matrix has rank $N-4$ with a clear spectral gap; and the first five gradients form a tight cluster.  The last signal is materially weaker than the first two: in a representative $N=7$ solve with a duality gap of order $10^{-8}$, their largest pairwise distance was $1.9\times10^{-4}$, about $0.2\%$ of the common gradient norm.  We therefore treat $g_0=\cdots=g_4$ as a numerically suggested ansatz, not as a consequence of optimality.

Our argument has three components.
\begin{enumerate}[leftmargin=2em]
\item Inspection of the Kim--Fessler dual certificate suggests a rank-one slack pattern; we adopt it as an ansatz to identify the balance condition $x_0-x_\star=L^{-1}\sum_{k=0}^N t_k\nabla f(x_k)$, which the construction then verifies directly.
\item We build a four-vertex, rank-three spherical seed at $N=7$.  Its exact entries are the solution of six linear equations in the algebraic numbers $t_0,\ldots,t_7$.  A short outward-rounded interval certificate verifies positive semidefiniteness and every omitted interpolation inequality.
\item We prove a spherical lifting lemma.  A common tilt of all existing unit gradients, followed by one new antipodal direction, preserves the old projection inequalities and makes all new ones automatic.  FGM's coefficient identities give exactly the tilt needed to preserve both the terminal balance and the initial radius.
\end{enumerate}

The resulting adversary is the conjugate of a strongly convex quadratic restricted to a polytope.  Equivalently, it is a smooth piecewise-quadratic Moreau envelope.  The novelty lies in selecting and propagating the FGM-compatible polytopes, not in this standard representation of the function.

\paragraph{Scope.} Our theorem concerns the points $x_k$ at which FGM actually evaluates $\nabla f(x_k)$.  It does not claim the same exact constant for the post-gradient points $y_{k+1}$.  In fact, the adversary constructed below satisfies $y_{N+1}=x_\star$ exactly (Remark~\ref{rem:post-gradient-hit}).  Proposition~\ref{prop:N1} settles $N=1$ separately; the analytical description of the five finite branches $2\leq N\leq6$ remains open.

\section{Problem statement and the known upper bound}

Let $\F_{0,L}(\R^d)$ be the class of convex differentiable functions whose gradients are $L$-Lipschitz.  FGM is written in the following form:
\begin{align}
 y_0&=x_0, \qquad t_0=1,\label{eq:fgm-init}\\
 y_{k+1}&=x_k-\frac1L\nabla f(x_k),\label{eq:fgm-y}\\
 t_{k+1}&=\frac{1+\sqrt{1+4t_k^2}}2,\label{eq:t-rec}\\
 x_{k+1}&=y_{k+1}+\frac{t_k-1}{t_{k+1}}(y_{k+1}-y_k).\label{eq:fgm-x}
\end{align}
We call $N$ the horizon: FGM performs $N$ update steps and queries gradients at the $N+1$ points $x_0,\ldots,x_N$.
The identities
\begin{equation}\label{eq:t-identities}
 t_k^2=t_{k-1}^2+t_k=\sum_{j=0}^k t_j
\end{equation}
follow directly from~\eqref{eq:t-rec}.

For a horizon $N$, define
\begin{equation}\label{eq:WN}
 W_N(L,R,d):=\sup_{\substack{f\in\F_{0,L}(\R^d),\ x_\star\in\arg\min f\\
                        \norm{x_0-x_\star}\leq R}}
                  \min_{0\leq k\leq N}\norm{\nabla f(x_k)}^2.
\end{equation}

Kim and Fessler's analytical dual certificate gives the following dimension-free upper bound.

\begin{theorem}[Kim--Fessler gradient bound]\label{thm:upper}
The following chain is Theorem~5.2, equation~(5.7), in Section~5.2 of the published version of~\cite{KimFessler2018}; the same display appears as equation~(5.7) in arXiv version~4.  Fix $N\geq1$, $L>0$, $R\geq0$, and $d\geq1$.  Let $f\in\F_{0,L}(\R^d)$ and $x_\star\in\arg\min f$ satisfy $\norm{x_0-x_\star}\leq R$, and let $x_0,\ldots,x_N$ and $y_0,\ldots,y_{N+1}$ be generated by~\eqref{eq:fgm-init}--\eqref{eq:fgm-x}, with~\eqref{eq:fgm-y} also applied at $k=N$ to define $y_{N+1}$.  Then
\begin{equation}\label{eq:upper}
 \min_{0\leq i\leq N+1}\norm{\nabla f(y_i)}
 \leq \min_{0\leq i\leq N}\norm{\nabla f(x_i)}
 \leq \frac{LR}{\sqrt{S_N}},
 \qquad S_N:=\sum_{k=0}^N t_k^2.
\end{equation}
Consequently,
$W_N(L,R,d)\leq L^2R^2/S_N$.
\end{theorem}

Inspection of the feasible dual point constructed in the proof of Theorem~\ref{thm:upper} (equations~(5.5)--(5.6) of~\cite{KimFessler2018}) suggests that the associated PSD slack is rank one, proportional to $\tau\tau^\top$ with $\tau=(t_0,\ldots,t_N,1)$.  Under that reading, complementary slackness predicts, in the normalized case $L=R=1$, the primal equality
\begin{equation}\label{eq:rank-one-balance}
 x_0-x_\star=\sum_{k=0}^N t_k g_k,
 \qquad g_k:=\nabla f(x_k).
\end{equation}
The positive multipliers attached to the performance constraints likewise predict $\norm{g_0}=\cdots=\norm{g_N}=S_N^{-1/2}$.  We emphasize that~\eqref{eq:rank-one-balance} and the equal-norm pattern enter only as an ansatz: no proof below relies on either claim about the structure of the dual certificate.  Our construction realizes both conditions while also satisfying every smooth-convex interpolation inequality omitted by the relaxed PEP.

We will also use the following exact identity associated with that rank-one pattern; it is an unconditional statement and does not depend on the dual-certificate reading above.  Besides removing any numerical ambiguity from the radius calculation at the seed, it exposes why the active faces below are the right ones.

\begin{lemma}[Rank-one radius identity]\label{lem:radius-identity}
Fix unit vectors $u_0,\ldots,u_N$, set $x_0=\sum_{k=0}^Nt_ku_k$, and generate formal FGM points with $L=1$ and prescribed gradients $u_k$.  If $z_k:=x_k-u_k$, then
\begin{equation}\label{eq:radius-identity}
 \frac{\norm{x_0}^2}{S_N}-1
 =\frac{2}{S_N}\sum_{i=1}^Nt_{i-1}^2
       \ip{z_{i-1}}{u_{i-1}-u_i}
  +\frac{2t_N^2}{S_N}\ip{z_N}{u_N}.
\end{equation}
In particular, vanishing adjacent and terminal inner products imply $\norm{x_0}^2=S_N$.
\end{lemma}

\begin{proof}
Put $z_{-1}:=y_0=x_0$ and, for $0\leq i\leq N$, define
\[
 \eta_i:=t_i z_i-(t_i-1)z_{i-1},
 \qquad \eta_{-1}:=x_0.
\]
Since $z_i=y_{i+1}=x_i-u_i$, the FGM recursion gives, for $i\geq1$,
\[
 z_i=z_{i-1}+\frac{t_{i-1}-1}{t_i}(z_{i-1}-z_{i-2})-u_i.
\]
Consequently $\eta_i=\eta_{i-1}-t_i u_i$ for $0\leq i\leq N$; the case
$i=0$ follows from $t_0=1$ and $z_0=x_0-u_0$.  Hence
\begin{equation}\label{eq:eta-balance}
 \eta_i=x_0-\sum_{k=0}^i t_k u_k.
\end{equation}
The assumed balance gives $\eta_N=0$.  Using
$\eta_{i-1}=\eta_i+t_i u_i$ and telescoping squared norms yields
\begin{align*}
 \norm{x_0}^2
 &=\sum_{i=0}^N\bigl(\norm{\eta_{i-1}}^2-\norm{\eta_i}^2\bigr)\\
 &=2\sum_{i=0}^N t_i\ip{\eta_i}{u_i}
   +\sum_{i=0}^N t_i^2\norm{u_i}^2.
\end{align*}
Because the $u_i$ are unit vectors,
\begin{equation}\label{eq:eta-norm-gap}
 \frac12\bigl(\norm{x_0}^2-S_N\bigr)
 =\sum_{i=0}^N t_i\ip{\eta_i}{u_i}.
\end{equation}
Finally, the definition of $\eta_i$ and
$t_i(t_i-1)=t_{i-1}^2$ give
\begin{align*}
 \sum_{i=0}^N t_i\ip{\eta_i}{u_i}
 &=\sum_{i=0}^N t_i^2\ip{z_i}{u_i}
   -\sum_{i=1}^N t_i(t_i-1)\ip{z_{i-1}}{u_i}\\
 &=\sum_{i=1}^N t_{i-1}^2
      \ip{z_{i-1}}{u_{i-1}-u_i}
   +t_N^2\ip{z_N}{u_N}.
\end{align*}
Combining this equality with~\eqref{eq:eta-norm-gap} and multiplying by
$2/S_N$ proves~\eqref{eq:radius-identity}.

For connection with the PEP certificate, set
$x_\star=u_\star=f_\star=0$, $f_i=\ip{u_i}{x_i}-1/2$, and
\[
 Q_{ij}:=f_i-f_j-\ip{u_j}{x_i-x_j}
              -\frac12\norm{u_i-u_j}^2.
\]
Then $Q_{\star,i}=0$,
$Q_{i-1,i}=\ip{z_{i-1}}{u_{i-1}-u_i}$, and
$Q_{N,\star}=\ip{z_N}{u_N}$.  Thus the calculation above is precisely the
specialized gap identity with multipliers $t_{i-1}^2$ and $t_N^2$; in
particular, the normalized terminal coefficient is $2t_N^2/S_N$.
\end{proof}

\section{The standard projection-envelope template}

We first record, for a self-contained proof, the standard function class used for the lower bound.

\begin{lemma}[Standard projection-envelope interpolant]\label{lem:projection-interpolant}
Let $K\subset\R^d$ be a nonempty compact convex set containing the origin, and define
\begin{equation}\label{eq:projection-function}
 \phi_K(x):=\max_{g\in K}\left\{\ip{x}{g}-\frac12\norm{g}^2\right\}.
\end{equation}
Then $\phi_K\in\F_{0,1}(\R^d)$, $\min\phi_K=\phi_K(0)=0$, and
\begin{equation}\label{eq:projection-moreau}
 \phi_K(x)=\min_z\left\{\sigma_K(z)+\frac12\norm{x-z}^2\right\}
 =\frac12\norm{x}^2-\frac12\dist(x,K)^2,
 \qquad \sigma_K(z):=\max_{g\in K}\ip{z}{g},
\end{equation}
as well as
\begin{equation}\label{eq:projection-gradient}
 \nabla\phi_K(x)=\Proj_K(x).
\end{equation}
Moreover, a vector $g_i\in K$ equals $\nabla\phi_K(x_i)$ if and only if
\begin{equation}\label{eq:projection-condition}
 \ip{x_i-g_i}{g_i-g}\geq0\qquad\text{for every }g\in K.
\end{equation}
When $K=\conv\{0,g_0,\ldots,g_N\}$, it suffices to check~\eqref{eq:projection-condition} at those vertices.
\end{lemma}

\begin{proof}
The conjugate of $\phi_K$ is $\delta_K(g)+\norm{g}^2/2$, which is $1$-strongly convex on its domain.  Hence $\phi_K$ is convex and $1$-smooth.  Since $\sigma_K^*=\delta_K$, the standard infimal-convolution identity gives the first equality in~\eqref{eq:projection-moreau}; completing the square gives its second equality and shows that the unique maximizer in~\eqref{eq:projection-function} is $\Proj_K(x)$.  This proves~\eqref{eq:projection-gradient}.  Since $0\in K$, $\phi_K\geq0$ and $\phi_K(0)=0$.  Finally,~\eqref{eq:projection-condition} is the standard variational characterization of Euclidean projection.
\end{proof}

Lemma~\ref{lem:projection-interpolant} is included as a convenient interface, not as a new interpolation theorem.  The ball/Huber specialization, the smooth max-over-simplex extension of Drori and Taylor~\cite{DroriTaylor2022}, and the exact polytope/Moreau formulation of Ma and Chen~\cite{MaChen2026} are the closest precedents.  What remains to be constructed is the FGM-specific sequence of spherical polytopes for which all required projections occur simultaneously.

Consequently, the lower-bound problem is geometric: construct equal-norm vectors $g_i$ and FGM points $x_i$ for which the residual $x_i-g_i$ belongs to the normal cone of a polytope at $g_i$.

\section{Coefficient calculus for FGM}

We now normalize $L=R=1$; Section~\ref{sec:main} restores the scaling.  It is convenient to represent FGM points as fixed linear combinations of a prescribed gradient list.

For a horizon $N$, let $e_0,\ldots,e_N$ denote the standard basis of $\R^{N+1}$.  Define coefficient vectors $p_k,r_k\in\R^{N+1}$ by
\begin{align}
 p_0&=r_0=0,\label{eq:coeff-init}\\
 r_{k+1}&=p_k-e_k,\label{eq:coeff-r}\\
 p_{k+1}&=r_{k+1}+\frac{t_k-1}{t_{k+1}}(r_{k+1}-r_k).\label{eq:coeff-p}
\end{align}
If $x_0$ and $g_0,\ldots,g_N$ are fixed, then the FGM relations are equivalent to
\begin{equation}\label{eq:x-coeff}
 x_k=x_0+\sum_{j=0}^N(p_k)_j g_j,
 \qquad y_{k+1}=x_0+\sum_{j=0}^N(r_{k+1})_j g_j.
\end{equation}

The construction will set $g_0=\cdots=g_4$.  For $N\geq4$, put $m_N=N-3$ and define $E_N\in\R^{(N+1)\times m_N}$ by
\begin{equation}\label{eq:merge-map}
 E_N^\top e_j=\begin{cases}
 e_0,&0\leq j\leq4,\\
 e_{j-4},&5\leq j\leq N.
 \end{cases}
\end{equation}
Here and below, standard basis vectors on the right live in the reduced space $\R^{m_N}$.  Let $\iota(j)=0$ for $j\leq4$ and $\iota(j)=j-4$ otherwise.  Define
\begin{align}
 w_N&:=E_N^\top(t_0,\ldots,t_N)^\top,\label{eq:wN}\\
 a_{i,N}&:=w_N+E_N^\top p_i,\label{eq:a-iN}\\
 d_{i,N}&:=a_{i,N}-e_{\iota(i)},\label{eq:d-iN}\\
 c_N&:=d_{N,N},\qquad C_N:=\one^\top c_N.\label{eq:cN}
\end{align}
If $v_0,\ldots,v_{m_N-1}$ are reduced gradient directions and $x_0=\sum_j (w_N)_jv_j$, then $a_{i,N}$ are the coefficients of $x_i$ and $d_{i,N}$ are the coefficients of $x_i-v_{\iota(i)}$.

The identities needed by the lift are elementary consequences of~\eqref{eq:coeff-init}--\eqref{eq:coeff-p}.  We record them together.

\begin{lemma}[FGM coefficient identities]\label{lem:coefficient-identities}
Let $N\geq4$, $T_N:=\sum_{k=0}^N t_k=t_N^2$, and $S_N:=\sum_{k=0}^N t_k^2$.  Then:
\begin{align}
 c_N&=\left(1-\frac1{t_N}\right)d_{N-1,N},\label{eq:terminal-proportionality}\\
 c_{N+1}&=\begin{pmatrix}
 \left(1-\frac1{t_{N+1}}\right)c_N\\ t_{N+1}-1
 \end{pmatrix},\label{eq:c-recursion}\\
 C_N&=\frac{T_N^2-S_N}{2T_N},\label{eq:C-closed}\\
 \one^\top d_{i,N}&\geq C_N\qquad(0\leq i\leq N),\label{eq:coefficient-monotonicity}\\
 d_{i,N+1}&=\begin{pmatrix}d_{i,N}\\ t_{N+1}\end{pmatrix}
 \qquad(0\leq i\leq N).\label{eq:old-coefficient-extension}
\end{align}
The vectors in~\eqref{eq:c-recursion} use the natural embedding $\R^{m_N}\subset\R^{m_N+1}$.
\end{lemma}

\begin{proof}
Write $h_i:=-\one^\top p_i$, the sum of the fixed-step coefficients used to form $x_i$.  From~\eqref{eq:coeff-r}--\eqref{eq:coeff-p}, $h_0=0$, $h_1=1$, and
\[
 h_{i+1}=h_i+1+\frac{t_i-1}{t_{i+1}}(h_i-h_{i-1})\qquad(i\geq1).
\]
Thus $h_i$ is increasing.  Since $\one^\top d_{i,N}=T_N-h_i-1$, this proves~\eqref{eq:coefficient-monotonicity} once the formula for $C_N$ is established.  A direct induction using $t_i^2=t_{i-1}^2+t_i$ gives
\[
 h_i+1=\frac{T_i^2+S_i}{2T_i},
\]
and therefore $C_N=T_N-h_N-1=(T_N^2-S_N)/(2T_N)$.  The same two-term recurrence, applied componentwise before taking row sums, gives~\eqref{eq:terminal-proportionality}.

For completeness, write $p_i^{[N]}$ when the recurrence is represented in $\R^{N+1}$.  Up to index $i$, it uses only $e_0,\ldots,e_{i-1}$, so $p_i^{[N+1]}=(p_i^{[N]},0)$ for $i\leq N$.  The merge maps consequently give
\[
 w_{N+1}=(w_N,t_{N+1}),\qquad
 E_{N+1}^\top p_i^{[N+1]}=(E_N^\top p_i^{[N]},0).
\]
The reduced basis vector $e_{\iota(i)}$ is embedded as $(e_{\iota(i)},0)$; subtracting it proves~\eqref{eq:old-coefficient-extension}.  Applying~\eqref{eq:terminal-proportionality} at horizon $N+1$ and then~\eqref{eq:old-coefficient-extension} with $i=N$ gives
\[
 c_{N+1}=\left(1-\frac1{t_{N+1}}\right)d_{N,N+1}
 =\begin{pmatrix}
   (1-t_{N+1}^{-1})c_N\\ t_{N+1}-1
  \end{pmatrix},
\]
which is~\eqref{eq:c-recursion}.  These calculations are also reproduced symbolically in the supplementary script \texttt{fgm\_recursive\_adversary.py}.
\end{proof}

\section{The algebraic seed at horizon seven}

We next give a finite exact seed.  All numbers in this section are algebraic, since they are obtained from~\eqref{eq:t-rec} and linear equations.

The merged prefix $g_0=\cdots=g_4$ is an ansatz suggested by the numerical clustering described in the introduction; it is not forced by the PEP.  At a horizon $H\geq5$, this ansatz leaves $m=H-3$ distinct unit directions and hence $m(m-1)/2$ free off-diagonal Gram entries.  Closure $G_Hc_H=0$ gives $m$ nominal affine equations.  In Lemma~\ref{lem:radius-identity}, the first four adjacent terms vanish because of the merged prefix, while closure and~\eqref{eq:terminal-proportionality} give $z_H=z_{H-1}=0$.  The remaining active-face equalities are indexed by $i=5,\ldots,H-1$, hence contribute $H-5=m-2$ equations.  The nominal count is therefore $2m-2$, and
\[
 \frac{m(m-1)}2\geq2m-2
 \quad\Longleftrightarrow\quad (m-1)(m-4)\geq0.
\]
Thus $m=4$, or $H=7$, is the first non-overdetermined horizon in this five-prefix construction.  At $H=7$ the system is square, and the certified nonzero determinant below proves that its six equations are independent.  This count explains the threshold inside our ansatz; it is not an impossibility proof for other ansatzes or other matching instances at $H\leq6$.

For $N=7$, the reduced directions are $v_0=g_0=\cdots=g_4$, $v_1=g_5$, $v_2=g_6$, and $v_3=g_7$.  In particular,
\[
 w_7=(t_4^2,t_5,t_6,t_7)^\top
     \approx(10.8562321,3.8326014,4.3650787,4.8936218)^\top.
\]
For orientation, the coefficient recurrence gives
\[
\begin{array}{c|rrrr}
i &(d_{i,7})_0&(d_{i,7})_1&(d_{i,7})_2&(d_{i,7})_3\\ \hline
0&9.8562321&3.8326014&4.3650787&4.8936218\\
4&4.1916291&3.8326014&4.3650787&4.8936218\\
5&3.0979518&2.8326014&4.3650787&4.8936218
\end{array}
\]
The decimals are illustrative; the exact values are defined by~\eqref{eq:coeff-init}--\eqref{eq:cN}.  Let $G_7\in\mathbb{S}^4$ be the symmetric matrix with unit diagonal whose six off-diagonal entries are uniquely determined by
\begin{align}
 G_7c_7&=0,\label{eq:seed-null}\\
 d_{4,7}^\top G_7(e_0-e_1)&=0,\label{eq:seed-face1}\\
 d_{5,7}^\top G_7(e_1-e_2)&=0.\label{eq:seed-face2}
\end{align}
Equations~\eqref{eq:seed-null}--\eqref{eq:seed-face2} form a $6\times6$ linear system.  For orientation only, its solution is
\begin{equation}\label{eq:G7-numeric}
G_7\approx
\begin{pmatrix}
1& .3936432000&-.1082450940&-.5892518844\\
.3936432000&1&-.0145831950&-.6283114426\\
-.1082450940&-.0145831950&1&-.6283114426\\
-.5892518844&-.6283114426&-.6283114426&1
\end{pmatrix}.
\end{equation}
The exact definition is the linear system, not the displayed decimals.

\begin{figure}[t]
\centering
\begin{tikzpicture}[scale=1.85,
  vertex/.style={circle,fill=blue!65!black,inner sep=1.5pt},
  edge/.style={gray!65}, residual/.style={->,very thick,orange!85!black}]
\coordinate (o)  at ( 0.0000, 0.0000);
\coordinate (v0) at ( 1.0000, 0.0000);
\coordinate (v1) at ( 0.3936, 0.9193);
\coordinate (v2) at ( 0.2395, 0.2789);
\coordinate (v3) at (-0.8284,-0.6020);
\foreach \a in {v0,v1,v2,v3}{\draw[edge] (o)--(\a);}
\draw[edge,dashed] (v0)--(v1)--(v2)--(v3)--(v0);
\draw[residual] (v0)--++(0.6730,0.4487)
 node[right] {$z_4\perp(v_0-v_1)$};
\draw[residual] (v1)--++(0.4216,0.3064)
 node[above] {$z_5\perp(v_1-v_2)$};
\fill[black] (o) circle (1.2pt) node[below right] {$0$};
\foreach \a/\lab in {v0/v_0,v1/v_1,v2/v_2,v3/v_3}
  {\node[vertex,label=above:$\lab$] at (\a) {};}
\end{tikzpicture}
\caption{An oblique projection of the rank-three seed at $N=7$.  The first five queried gradients share direction $v_0$; the remaining directions are $v_1,v_2,v_3$.  The arrows show the two active neighboring-face residuals in the seed equations.  Orthogonality holds in the ambient three-dimensional realization and need not appear as a right angle after projection.}
\label{fig:seed}
\end{figure}
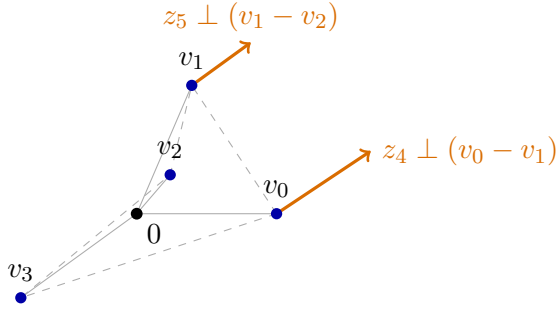

\begin{lemma}[Certified seed]\label{lem:seed}
The matrix $G_7$ exists uniquely, is positive semidefinite of rank three, and satisfies
\begin{align}
 d_{i,7}^\top G_7\bigl(e_{\iota(i)}-e_j\bigr)&\geq0
 &&(0\leq i\leq7,\ 0\leq j\leq3),\label{eq:seed-projection-vertices}\\
 d_{i,7}^\top G_7e_{\iota(i)}&\geq0
 &&(0\leq i\leq7),\label{eq:seed-projection-zero}\\
 w_7^\top G_7w_7&=S_7.\label{eq:seed-radius}
\end{align}
Moreover, $G_7d_{6,7}=G_7d_{7,7}=0$.
\end{lemma}

\begin{proof}
The determinant of the $6\times6$ system~\eqref{eq:seed-null}--\eqref{eq:seed-face2} lies in
\[
[221.65769092133139765209672580134623832952,
 221.65769092133139765209672580134623832953],
\]
so the solution is unique.  Substitution with outward-rounded intervals shows that every principal minor of order at most three is positive; the smallest lower endpoint is $0.16613798382964938$.  Since $G_7c_7=0$ with $c_7\neq0$, $\det G_7=0$, hence $G_7\succeq0$ and $\operatorname{rank}G_7=3$.

The same substitution verifies~\eqref{eq:seed-projection-vertices}--\eqref{eq:seed-projection-zero}.  After removing identities that are zero by construction, the smallest positive residual has lower endpoint $0.05680185218226115$.  By~\eqref{eq:terminal-proportionality} and~\eqref{eq:seed-null}, $G_7d_{6,7}=G_7d_{7,7}=0$.  The adjacent terms in Lemma~\ref{lem:radius-identity} vanish: the first four because $v_0=\cdots=v_4$, the next two by~\eqref{eq:seed-face1}--\eqref{eq:seed-face2}, and the last because $G_7d_{6,7}=0$.  The terminal term vanishes because $G_7d_{7,7}=0$.  Hence Lemma~\ref{lem:radius-identity} proves~\eqref{eq:seed-radius} exactly.  The self-contained script \texttt{fgm\_base7\_rational\_certificate.py} uses only exact rational endpoints and integer square roots; a high-precision secondary \texttt{mpmath.iv} replay is provided in \texttt{fgm\_base7\_interval\_certificate.py}.
\end{proof}

Every $t_k$ is algebraic by~\eqref{eq:t-rec}, and the nonzero determinant above makes each entry of $G_7$ an element of the algebraic field $\mathbb{Q}(t_1,\ldots,t_7)$.  The certificate evaluates rational outward-rounded enclosures for these elements.  An enclosure excluding zero fixes the sign of the enclosed nonzero algebraic number.  Thus the lower bounds $0.166137983829$ for the required principal minors, $0.056801852182$ for the strict projection margins, and $4.789612460337$ for the base-lift gap $C_7-t_8$ constitute rigorous sign certificates, rather than floating-point evidence.  Every arithmetic operation used by the certificate is inclusion-isotone, and division is performed only after its denominator interval has been certified to exclude zero; consequently, every printed interval contains the exact algebraic quantity.  This is only a fixed-dimensional sign check; no horizon-dependent SDP or numerical optimization appears in the theorem.

\section{An FGM-compatible spherical cone lift}\label{sec:lift}

The pyramid operation underlying the next lemma is elementary.  Its role here is the more specific one of preserving all old normal-cone inequalities under a horizon extension whose scalar parameter is dictated by FGM.

\begin{lemma}[Projection-preserving spherical lift]\label{lem:spherical-lift}
Let $v_0,\ldots,v_{m-1}$ be unit vectors with Gram matrix $G$, let residuals $z_i=\sum_j d_{ij}v_j$ carry vertex labels $\ell(i)\in\{0,\ldots,m-1\}$, and suppose
\begin{equation}\label{eq:old-proj}
 \ip{z_i}{v_{\ell(i)}-v_j}\geq0,\qquad
 \ip{z_i}{v_{\ell(i)}}\geq0
\end{equation}
for all $i,j$.  Suppose $z_{\rm last}=0$, $D_i:=\sum_jd_{ij}\geq C>0$, and $D_{\rm last}=C$.

Choose $s\in(0,1)$ and a unit vector $e$ orthogonal to all $v_j$.  Put $a=\sqrt{1-s^2}$ and
\begin{equation}\label{eq:lift-vectors}
 v_j'=av_j+se\quad(0\leq j<m),
 \qquad v_m'=-e.
\end{equation}
For $\theta=sC$, define lifted old residuals
\begin{equation}\label{eq:lift-residual}
 z_i'=az_i+(sD_i-\theta)e.
\end{equation}
Then every $z_i'$ satisfies the projection inequalities against $0,v_0',\ldots,v_m'$.  If $i={\rm last}$, then $z_i'=0$.
\end{lemma}

\begin{proof}
Set $\delta_i=s(D_i-C)\geq0$.  For an old vertex $v_j'$,
\[
 \ip{z_i'}{v_{\ell(i)}'-v_j'}=a^2\ip{z_i}{v_{\ell(i)}-v_j}\geq0.
\]
The inequalities against the origin and the new vertex are, respectively,
\begin{align*}
 \ip{z_i'}{v_{\ell(i)}'}
 &=a^2\ip{z_i}{v_{\ell(i)}}+s\delta_i\geq0,\\
 \ip{z_i'}{v_{\ell(i)}'-v_m'}
 &=a^2\ip{z_i}{v_{\ell(i)}}+(1+s)\delta_i\geq0.
\end{align*}
For the last residual, both $z_{\rm last}=0$ and $D_{\rm last}=C$ hold, so~\eqref{eq:lift-residual} is zero.
\end{proof}

At the Gram level,~\eqref{eq:lift-vectors} is
\begin{equation}\label{eq:gram-lift}
G'=
\begin{pmatrix}
(1-s^2)G+s^2\one\one^\top&-s\one\\
-s\one^\top&1
\end{pmatrix}.
\end{equation}
The vector realization proves $G'\succeq0$ without a Schur-complement calculation.

\section{Exact worst-case performance}\label{sec:main}

We now recursively construct the adversary.  Start from $G_7$.  Given $G_N$ and $c_N$, define
\begin{equation}\label{eq:sN}
 s_N:=\frac{t_{N+1}}{C_N},\qquad C_N=\one^\top c_N,
\end{equation}
and obtain $G_{N+1}$ from~\eqref{eq:gram-lift} with $s=s_N$.

\begin{lemma}[Well-defined recursion]\label{lem:well-defined}
For every $N\geq7$, $0<s_N<1$.  The matrix $G_N$ is the Gram matrix of $m_N=N-3$ unit vectors spanning a space of dimension $N-4$, and
\begin{align}
G_Nc_N&=0,\label{eq:GN-null}\\
w_N^\top G_Nw_N&=S_N,\label{eq:GN-radius}\\
d_{i,N}^\top G_N(e_{\iota(i)}-e_j)&\geq0,\label{eq:GN-proj-old}\\
d_{i,N}^\top G_Ne_{\iota(i)}&\geq0.\label{eq:GN-proj-zero}
\end{align}
The projection inequalities hold for all compatible $i,j$.
\end{lemma}

\begin{proof}
The base case is Lemma~\ref{lem:seed}.  The rational-interval certificate gives the rigorous lower bound $C_7-t_8>4.789612460337$, hence $s_7<1$.  If $C_N>t_{N+1}$, then~\eqref{eq:c-recursion} gives
\[
C_{N+1}=\left(1-\frac1{t_{N+1}}\right)C_N+t_{N+1}-1
>2t_{N+1}-2>t_{N+2},
\]
where the final inequality uses $t_{N+1}>3$ and $t_{N+2}<t_{N+1}+1$.  Thus $s_N\in(0,1)$ for all $N\geq7$.

Let $v_j'$ be the lift~\eqref{eq:lift-vectors}.  By~\eqref{eq:c-recursion},
\begin{align*}
\sum_{j=0}^{m_N-1}\left(1-\frac1{t_{N+1}}\right)(c_N)_jv_j'
 +(t_{N+1}-1)v_{m_N}'
&=\left(1-\frac1{t_{N+1}}\right)a\sum_j(c_N)_jv_j\\
&\quad+\left[\left(1-\frac1{t_{N+1}}\right)s_NC_N-(t_{N+1}-1)\right]e\\
&=0.
\end{align*}
This proves~\eqref{eq:GN-null} at the next horizon.

For an old query $i\leq N$, identity~\eqref{eq:old-coefficient-extension} says that its new coefficient vector is $(d_{i,N},t_{N+1})$.  Since the old directions are tilted to $av_j+s_Ne$ and the new direction is $-e$, its new-horizon residual is therefore
\begin{equation}\label{eq:old-residual-transform}
 z_{i,N+1}=az_{i,N}+\left(s_N\one^\top d_{i,N}-t_{N+1}\right)e
 =az_{i,N}+s_N\left(\one^\top d_{i,N}-C_N\right)e.
\end{equation}
Lemma~\ref{lem:coefficient-identities} and Lemma~\ref{lem:spherical-lift} therefore prove all old-query inequalities against all old and new vertices.  The old terminal residual remains zero, and~\eqref{eq:GN-null} makes the new terminal residual zero.  Hence~\eqref{eq:GN-proj-old}--\eqref{eq:GN-proj-zero} propagate.

It remains to verify the radius.  Under the lift,
\[
u_{N+1}:=\sum_{k=0}^{N+1}t_kv_{\iota(k)}'
=au_N+(s_NT_N-t_{N+1})e.
\]
By~\eqref{eq:C-closed},
\[
s_N=\frac{2T_Nt_{N+1}}{T_N^2-S_N}.
\]
Therefore
\begin{align*}
\norm{u_{N+1}}^2
&=(1-s_N^2)S_N+(s_NT_N-t_{N+1})^2\\
&=S_N+t_{N+1}^2+s_N\bigl[s_N(T_N^2-S_N)-2T_Nt_{N+1}\bigr]\\
&=S_{N+1},
\end{align*}
which proves~\eqref{eq:GN-radius}.  Finally, since $a=\sqrt{1-s_N^2}>0$ and $v_{m_N}'=-e$,
\[
 v_j=a^{-1}(v_j'+s_Nv_{m_N}')\qquad(0\leq j<m_N).
\]
Hence
$\operatorname{span}\{v_0',\ldots,v_{m_N}'\}
=\operatorname{span}\{v_0,\ldots,v_{m_N-1},e\}$.
The vector $e$ is orthogonal to the old span, so its dimension, and therefore the Gram rank, increases from $N-4$ to $N-3$.
\end{proof}

We can now state the main result.

\begin{theorem}[Exact best queried-gradient rate of FGM]\label{thm:main}
Let $N\geq7$, $L>0$, $R\geq0$, and $d\geq N-4$.  For the iterates~\eqref{eq:fgm-init}--\eqref{eq:fgm-x},
\begin{equation}\label{eq:main-result}
 W_N(L,R,d)=\frac{L^2R^2}{S_N},
 \qquad S_N=\sum_{k=0}^N t_k^2.
\end{equation}
The lower bound is attained by a smooth convex piecewise-quadratic function in dimension $N-4$.
\end{theorem}

\begin{proof}
Theorem~\ref{thm:upper} supplies the upper bound.  If $R=0$, then $x_0=x_\star$, every queried gradient vanishes, and both sides of~\eqref{eq:main-result} are zero.  Hence assume $R>0$.  It remains to construct a matching instance.  First take $L=R=1$.  Realize $G_N$ by unit vectors $v_0,\ldots,v_{m_N-1}\in\R^{N-4}$, and set
\[
 q_N:=S_N^{-1/2},\qquad
 g_k:=q_Nv_{\iota(k)},\qquad
 K_N:=\conv\{0,g_0,\ldots,g_N\}.
\]
Define
\begin{equation}\label{eq:explicit-worst-function}
 f_N(x):=\max_{g\in K_N}\left\{\ip{x}{g}-\frac12\norm{g}^2\right\}
\end{equation}
and take $x_\star=0$ and $x_0=\sum_{k=0}^Nt_kg_k$.  By~\eqref{eq:GN-radius}, $\norm{x_0}=1$.

Generate formal points $x_i$ from~\eqref{eq:x-coeff} using the prescribed $g_i$.  Equations~\eqref{eq:GN-proj-old}--\eqref{eq:GN-proj-zero}, after multiplication by $q_N^2$, are exactly
\[
\ip{x_i-g_i}{g_i-g}\geq0\qquad(g\in\{0,g_0,\ldots,g_N\}).
\]
Lemma~\ref{lem:projection-interpolant} gives $\nabla f_N(x_i)=g_i$.  Hence the formal points are the actual FGM iterates by induction, and
\[
\min_{0\leq i\leq N}\norm{\nabla f_N(x_i)}^2
=\norm{g_i}^2=\frac1{S_N}.
\]

For general $L$ and $R$, set
\[
 f_{N,L,R}(x)=LR^2 f_N\left(\frac{x-x_\star}{R}\right).
\]
Then $f_{N,L,R}\in\F_{0,L}$, the FGM trajectory is the corresponding scaled trajectory, and every queried gradient has squared norm $L^2R^2/S_N$.  Embed the construction isometrically when $d>N-4$.
\end{proof}

\begin{remark}[The next post-gradient point hits the minimizer]\label{rem:post-gradient-hit}
Let $V_Na:=\sum_j a_jv_j$, so $G_N=V_N^\top V_N$.  Since $G_Nc_N=0$,
$\norm{V_Nc_N}^2=c_N^\top G_Nc_N=0$.  But $c_N=d_{N,N}$ is the reduced coefficient vector of the normalized terminal post-gradient residual; therefore
\[
 y_{N+1}=x_N-\nabla f_N(x_N)=q_NV_Nc_N=0=x_\star.
\]
After scaling, the same identity is
$y_{N+1}=x_N-L^{-1}\nabla f_{N,L,R}(x_N)=x_\star$.  Hence this worst-case instance has
$\min_{0\leq k\leq N+1}\norm{\nabla f(y_k)}=0$, even though it attains the queried-point bound.  Thus it cannot serve as a matching lower-bound instance for the post-gradient criterion.  This explains why the full Kim--Fessler chain~\eqref{eq:KF-upper-intro} contains two genuinely different performance questions and why Theorem~\ref{thm:main} is stated only for $x_0,\ldots,x_N$; it does not by itself determine the post-gradient worst-case value.
\end{remark}

\begin{remark}[Explicitness]
Formula~\eqref{eq:explicit-worst-function}, the six algebraic equations defining $G_7$, and the scalar recursion~\eqref{eq:gram-lift}--\eqref{eq:sN} specify the adversary without solving an SDP.  Coordinates can be obtained from any rank-revealing Cholesky or eigenvalue factorization of $G_N$.
\end{remark}

\begin{remark}[Why the construction is high-dimensional]
The seed has rank three and each lift adds one dimension, so the displayed adversary has dimension $N-4$.  This matches the stable rank observed in high-accuracy primal PEP solutions.  We have neither an analytical uniqueness result for the primal optimum nor a proof that $N-4$ is a dimension lower bound for all matching instances; the rank of one numerical solution does not establish either claim.
\end{remark}

\section{Numerical PEP audit and finite-horizon exceptions}

Numerical exact-interpolation PEP evidence for this performance measure predates the present work: Kim and Fessler reported tight numerical values at selected horizons in Section~6.2 and Table~3 of~\cite{KimFessler2018}.  The computations below are a reproducibility audit, add the transition horizons $N=5,6,7,8$, and report signed differences and primal--dual gaps; they are not offered as the first numerical evidence for equality.

Table~\ref{tab:numerics} compares dimension-unrestricted exact-interpolation PEP computations with the analytical upper bound for $L=R=1$; equivalently, the Gram dimension is allowed to be large enough for the finite interpolation data.  Here $\widehat W_N$ is the numerical reconstructed dual objective returned by PEPit/Clarabel and $\Delta_N:=\widehat W_N-1/S_N$.  The largest displayed primal--dual gap is $3.4\times10^{-9}$.  Neither the numerical objective nor the reported gap is a rigorous enclosure or certificate, and signs at or below this numerical scale have no mathematical significance.  Equality for $N\geq7$ follows from Theorem~\ref{thm:main}, not from the decimal outputs; the values for $2\leq N\leq6$ remain numerical.

\begin{table}[ht]
\centering
\caption{Exact-interpolation queried-point PEP computations for $L=R=1$.}
\label{tab:numerics}
\begin{tabular}{rcccc}
\toprule
$N$ & $\widehat W_N$ & $1/S_N$ & $\Delta_N$ & primal--dual gap \\
\midrule
1 & 0.2500000002 & 0.2763932023 & $-2.64\times10^{-2}$ & $4.7\times10^{-10}$\\
2 & 0.0928513212 & 0.1186296604 & $-2.58\times10^{-2}$ & $3.4\times10^{-9}$\\
3 & 0.0427220648 & 0.0625353814 & $-1.98\times10^{-2}$ & $1.2\times10^{-9}$\\
4 & 0.0291845018 & 0.0372478605 & $-8.06\times10^{-3}$ & $2.3\times10^{-10}$\\
5 & 0.0234677658 & 0.0240754932 & $-6.08\times10^{-4}$ & $2.8\times10^{-9}$\\
6 & 0.0164895982 & 0.0165043940 & $-1.48\times10^{-5}$ & $9.0\times10^{-10}$\\
7 & 0.0118290763 & 0.0118290756 & $+6.7\times10^{-10}$ & $7.4\times10^{-10}$\\
8 & 0.0087793168 & 0.0087793158 & $+9.9\times10^{-10}$ & $8.2\times10^{-10}$\\
10& 0.0052360222 & 0.0052360213 & $+8.9\times10^{-10}$ & $6.0\times10^{-10}$\\
15& 0.0019420477 & 0.0019420474 & $+3.0\times10^{-10}$ & $1.5\times10^{-10}$\\
\bottomrule
\end{tabular}
\end{table}

The first exceptional horizon admits a short exact description.

\begin{proposition}[The first horizon]\label{prop:N1}
For every $L>0$, $R\geq0$, and $d\geq1$,
\[
 W_1(L,R,d)=\frac{L^2R^2}{4}.
\]
\end{proposition}

\begin{proof}
The case $R=0$ is immediate.  Let $g_i=\nabla f(x_i)$ and
$x_1=x_0-L^{-1}g_0$.  Cocoercivity between $x_0$ and $x_1$ gives
$\ip{g_1}{g_0}\geq\norm{g_1}^2$, while cocoercivity between $x_1$ and a
minimizer gives
$\ip{g_1}{x_1-x_\star}\geq L^{-1}\norm{g_1}^2$.  Therefore
\[
 \ip{g_1}{x_0-x_\star}
 =\ip{g_1}{x_1-x_\star}+L^{-1}\ip{g_1}{g_0}
 \geq 2L^{-1}\norm{g_1}^2.
\]
Cauchy--Schwarz yields $\norm{g_1}\leq LR/2$, and hence the same upper
bound for the minimum of the first two gradient norms.  For equality, fix a
unit vector $e$, let $K=[0,(LR/2)e]$, and define
\[
 f(x)=\max_{g\in K}\left\{\ip{x}{g}-\frac{1}{2L}\norm{g}^2\right\}.
\]
With $x_\star=0$ and $x_0=Re$, the scaled form of
Lemma~\ref{lem:projection-interpolant} gives
$f\in\F_{0,L}(\R^d)$ and $\nabla f(x)=\Proj_K(Lx)$; hence
$\nabla f(Re)=(LR/2)e$, $x_1=(R/2)e$, and
$\nabla f(x_1)=(LR/2)e$.  Since $0\in K$, one also has $f(x)\geq0$
for every $x$ and $f(0)=0$, so $x_\star=0$ is indeed a minimizer.
\end{proof}

The horizons $2\leq N\leq6$ do not form a proved continuation of either
Proposition~\ref{prop:N1} or the lifted $N\geq7$ branch.  Within the
five-prefix ansatz, the nominal seed system is overdetermined before $N=7$;
this explains why that completion is generically overdetermined but does not exclude different
Gram structures.  We claim no closed forms, optimizer uniqueness, or rank
pattern for those five numerical optima.

The disagreement for small $N$ is useful methodologically.  The Kim--Fessler SDP is an outer relaxation, so its analytical dual solution alone does not imply equality.  Conversely, the tight numerical values already reported by Kim and Fessler, and reproduced here, do not give a symbolic all-horizon proof.  The new contribution is therefore not the numerical coincidence, but its realization by an explicit analytic family for every $N\geq7$.

For completeness, define the distinct post-gradient performance
\[
 Y_N(L,R,d):=\sup_{\substack{f\in\F_{0,L}(\R^d),\ x_\star\in\arg\min f\\
                     \norm{x_0-x_\star}\leq R}}
       \min_{0\leq k\leq N+1}\norm{\nabla f(y_k)}^2
\]
under the same smoothness and radius constraints.  Table~\ref{tab:y-numerics}
reports a separate, dimension-unrestricted exact-interpolation PEP audit (the
Gram dimension is not fixed).  The numerical values lie well below $1/S_N$,
but no analytical formula is claimed; neither the numerical dual objectives
nor the reported gaps are rigorous enclosures.

\begin{table}[ht]
\centering
\caption{Dimension-unrestricted post-gradient PEP computations for $L=R=1$; $\widehat Y_N$ is the numerical reconstructed dual objective.}
\label{tab:y-numerics}
\begin{tabular}{rccc}
\toprule
$N$ & $\widehat Y_N$ & $1/S_N$ & primal--dual gap\\
\midrule
1 & 0.1111111112 & 0.2763932023 & $1.4\times10^{-10}$\\
2 & 0.0545452170 & 0.1186296604 & $2.6\times10^{-10}$\\
3 & 0.0293398935 & 0.0625353814 & $1.8\times10^{-9}$\\
4 & 0.0170223951 & 0.0372478605 & $2.7\times10^{-10}$\\
5 & 0.0105015504 & 0.0240754932 & $3.3\times10^{-11}$\\
6 & 0.0068110769 & 0.0165043940 & $3.5\times10^{-11}$\\
7 & 0.0046033545 & 0.0118290756 & $5.5\times10^{-11}$\\
8 & 0.0036255863 & 0.0087793158 & $2.9\times10^{-10}$\\
10& 0.0026903503 & 0.0052360213 & $2.3\times10^{-10}$\\
15& 0.0012524817 & 0.0019420474 & $3.2\times10^{-10}$\\
\bottomrule
\end{tabular}
\end{table}

\section{A conditional propagation theorem}

A rank-one relaxed-PEP slack is a useful discovery signal, but it is not by
itself sufficient for an exact interpolation or an all-horizon construction.
Nor is recursion across horizons new in itself: resisting-oracle arguments
grow hidden orthogonal subspaces, and recursive PEP-style certificate gluing
appears in~\cite{AltschulerParrilo2025}.  The narrower reusable statement
established by our argument is the following conditional primal theorem,
whose hypotheses can be checked directly from an algorithm's coefficient
arrays.

\begin{proposition}[Conditional propagation of spherical adversaries]
\label{prop:conditional-propagation}
Consider a homogeneous fixed-step first-order scheme in the normalized case
$L=1$, where homogeneous means that a common positive scaling of all points
and prescribed oracle vectors preserves the scheme's linear trajectory
equations.  The horizon extension considered here retains the old queries and
adds exactly one new terminal query, labelled by one new oracle direction.  At
horizon $n$, suppose its formal trajectory under prescribed unit
oracle vectors $v_0,\ldots,v_{m-1}$ has the representation below, where every
query label satisfies $\ell(i)\in\{0,\ldots,m-1\}$:
\[
 x_{0,n}=\sum_{j=0}^{m-1}(w_n)_jv_j,
 \qquad
 x_{i,n}=v_{\ell(i)}+z_{i,n},
 \qquad
 z_{i,n}=\sum_{j=0}^{m-1}(d_{i,n})_jv_j.
\]
Set $c_n=d_{n,n}$, $C_n=\one^\top c_n$,
$D_{i,n}=\one^\top d_{i,n}$, and $A_n=\one^\top w_n$.  Assume:
\begin{enumerate}[label=\textup{(\roman*)},leftmargin=2.5em]
\item for every old query index $i\leq n$ and every old vertex index $j<m$,
the corresponding residual satisfies
$\ip{z_{i,n}}{v_{\ell(i)}-v_j}\geq0$ and
$\ip{z_{i,n}}{v_{\ell(i)}}\geq0$;
\item $z_{n,n}=0$, $C_n>0$, and $D_{i,n}\geq C_n$ for every $i\leq n$;
\item for some $\theta_n\in(0,C_n)$ and $\rho_n>0$, the next coefficient table is
\[
 w_{n+1}=(w_n,\theta_n),\qquad
 d_{i,n+1}=(d_{i,n},\theta_n)\ (i\leq n),\qquad
 d_{n+1,n+1}=\rho_n(c_n,\theta_n),
\]
with the old labels unchanged and the new query labelled by the new
direction.
\end{enumerate}
Let $s_n=\theta_n/C_n$, $a_n=\sqrt{1-s_n^2}$, choose a unit
$e\perp\operatorname{span}\{v_j\}$, and set
\[
 v_j^+=a_nv_j+s_ne\quad(j<m),
 \qquad v_m^+=-e.
\]
Then $v_0^+,\ldots,v_m^+$ are unit vectors, and all horizon-$(n+1)$ residuals satisfy the projection inequalities
against $0,v_0^+,\ldots,v_m^+$, and the new terminal residual is zero.
Moreover, if $\ker G_n=\operatorname{span}\{c_n\}$, the new Gram matrix is
\[
 G_{n+1}=
 \begin{pmatrix}
 (1-s_n^2)G_n+s_n^2\one\one^\top&-s_n\one\\
 -s_n\one^\top&1
 \end{pmatrix},
\]
with $\ker G_{n+1}=\operatorname{span}\{d_{n+1,n+1}\}$ and
$\operatorname{rank}G_{n+1}=\operatorname{rank}G_n+1$.  Finally, writing
$\mathcal{R}_n=\norm{x_{0,n}}^2$, the radius obeys
\begin{equation}\label{eq:abstract-radius-recursion}
 \mathcal{R}_{n+1}=(1-s_n^2)\mathcal{R}_n+(s_nA_n-\theta_n)^2.
\end{equation}
Define
$K_{n+1}:=\conv\{0,v_0^+,\ldots,v_m^+\}$.  The projection interpolant of
Lemma~\ref{lem:projection-interpolant} therefore turns the lifted data into a
global smooth convex instance.  When $\mathcal{R}_{n+1}>0$, homogeneous scaling to unit
initial radius with $\alpha=\mathcal{R}_{n+1}^{-1/2}$, $\widetilde K_{n+1}=\alpha
K_{n+1}$, and $\widetilde x_{i,n+1}=\alpha x_{i,n+1}$ gives
\[
 \min_{0\leq i\leq n+1}
 \norm{\nabla\phi_{\widetilde K_{n+1}}(\widetilde x_{i,n+1})}^2
 =\frac1{\mathcal{R}_{n+1}}.
\]
Consequently, an analytical upper bound equal to $1/\mathcal{R}_{n+1}$ is exact when it
concerns this same queried-gradient performance measure, method, function
class, radius normalization, and an ambient dimension admitting the displayed
realization.
\end{proposition}

\begin{proof}
For an old query, the one-coordinate extension gives
\[
 z_{i,n+1}
 =a_nz_{i,n}+(s_nD_{i,n}-\theta_n)e
 =a_nz_{i,n}+s_n(D_{i,n}-C_n)e.
\]
This is exactly the residual in Lemma~\ref{lem:spherical-lift}; its three
projection inequalities follow from assumptions (i)--(ii).  The new terminal
residual equals
\[
 \rho_n\left(
 a_n\sum_j(c_n)_jv_j+(s_nC_n-\theta_n)e
 \right)=0,
\]
so its inequalities hold with equality.

The displayed $G_{n+1}$ is the Gram matrix of the lifted directions, hence is
positive semidefinite.  If $\sum_{j<m}\xi_jv_j^++\eta v_m^+=0$, orthogonal
decomposition gives $\sum_j\xi_jv_j=0$ and
$s_n\one^\top\xi-\eta=0$.  The one-dimensional old kernel implies
$\xi=\lambda c_n$ and $\eta=\lambda \theta_n$.  Thus the new kernel is spanned by
$(c_n,\theta_n)$, equivalently by $d_{n+1,n+1}$, and the rank increases by one.

Finally,
$x_{0,n+1}=a_nx_{0,n}+(s_nA_n-\theta_n)e$; orthogonality proves
\eqref{eq:abstract-radius-recursion}.  The already verified variational
inequalities identify every prescribed vertex as the Euclidean projection at
its formal query point.  The prescribed oracle vectors are therefore the
actual gradients of the global function from
Lemma~\ref{lem:projection-interpolant}; induction through the method equations
identifies the formal and actual trajectories.  If $\mathcal{R}_{n+1}>0$, set
$\alpha=\mathcal{R}_{n+1}^{-1/2}$ and replace the polytope, all points, and all oracle
vectors by their $\alpha$-scaled copies.  Projection equivariance and the
scheme's homogeneity preserve the interpolation and trajectory, while the
initial radius becomes one and every gradient norm becomes $\alpha$.
\end{proof}

FGM satisfies all hypotheses with
$\theta_N=t_{N+1}$, $\rho_N=1-t_{N+1}^{-1}$, and $A_N=T_N$.  Specifically,
\eqref{eq:old-coefficient-extension} is the old-row extension,
\eqref{eq:c-recursion} is the terminal extension, and
\eqref{eq:coefficient-monotonicity} gives the coefficient-sum condition.
The proof of Lemma~\ref{lem:well-defined} verifies $\theta_N=t_{N+1}<C_N$ and reduces
\eqref{eq:abstract-radius-recursion} to $S_{N+1}=S_N+t_{N+1}^2$.  Thus the
proposition packages the precise interface between the finite seed and the
all-horizon construction.  It does not assert that a rank-one dual slack
automatically supplies such a seed or coefficient extension.

\section{Reproducibility and limitations}

The supplementary material contains:
\begin{itemize}[leftmargin=2em]
\item pinned direct Python requirements and the exact-interpolation PEP script for the queried and post-gradient metrics;
\item \texttt{fgm\_base7\_rational\_certificate.py}, which rebuilds the algebraic $6\times6$ seed system using only \texttt{Fraction} endpoints and integer square roots, together with a high-precision secondary \texttt{mpmath.iv} replay;
\item \texttt{fgm\_recursive\_adversary.py}, which generates $G_N$ and checks Gram rank, radius, terminal balance, and every projection inequality (tested through $N=100$);
\end{itemize}

Two limitations should be explicit.  First, no analytical characterization is given for the five exceptional horizons $2\leq N\leq6$.  Second, Theorem~\ref{thm:main} concerns the minimum over queried points, not the terminal queried gradient or the post-gradient sequence.  Both variants exhibit different numerical branch changes and require separate analyses.

\section*{Use of AI assistance}

Parts of this work were developed with the assistance of \textsc{ChatGPT} (GPT-5.6 sol), a
proprietary large language model.

\ It was used to explore candidate seed structures at horizon seven, to
suggest the common-latitude lifting ansatz underlying
Lemma~\ref{lem:spherical-lift}, and to draft the interval-arithmetic code of
the supplementary certificates.

The author directed the investigation, verified every mathematical statement,
and takes full responsibility for the contents of this paper.  Two features of
the argument are relevant to that verification.  First, the horizon-seven seed
is certified by \texttt{fgm\_base7\_rational\_certificate.py}, which uses only
\texttt{Fraction} arithmetic and integer square roots, rounds every operation
outward, and refuses to divide by an interval that has not been certified to
exclude zero; its conclusions are therefore mathematical enclosures rather than
floating-point evidence, and they are independently reproducible in a fraction
of a second.  Second, the lifting argument of
Sections~\ref{sec:lift}--\ref{sec:main} is a symbolic proof: it is verified by
hand and does not depend on any numerical or machine-generated output.  No
claim in this paper rests on an unverified model output.

\section{Conclusion}

For standard FGM, a classical relaxed PEP upper bound is exactly attainable for every horizon $N\geq7$.  Using a standard projection-envelope function family, we construct an FGM-specific recursively lifted spherical polytope that gives a matching global instance.  This establishes the sharp finite-horizon constant $L^2R^2/S_N$, supplies an explicit worst case in dimension $N-4$, and upgrades previously reported numerical PEP equality to an all-horizon theorem.  Proposition~\ref{prop:conditional-propagation} isolates the precise additional hypotheses under which this particular primal lift propagates a fixed spherical seed; it does not claim novelty for Moreau-envelope adversaries, dimension extension, generic horizon recursion, or pyramid constructions, and a rank-one relaxed-PEP certificate alone is not sufficient.

\appendix
\section{Seed certificate data}

For convenience, the outward-rounded enclosures for the six off-diagonal entries of $G_7$, in the order $(G_{01},G_{02},G_{03},G_{12},G_{13},G_{23})$, are
\begin{align*}
G_{01}&\in[\phantom{-}0.3936432000209865,\ \phantom{-}0.3936432000209866],\\
G_{02}&\in[-0.1082450940075492,\ -0.1082450940075491],\\
G_{03}&\in[-0.5892518844114606,\ -0.5892518844114605],\\
G_{12}&\in[-0.0145831949772622,\ -0.0145831949772621],\\
G_{13}&\in[-0.6283114425647131,\ -0.6283114425647130],\\
G_{23}&\in[-0.6283114425647131,\ -0.6283114425647130].
\end{align*}
The equality $G_{13}=G_{23}$ follows from the exact linear system.  The supplementary rational-interval script recomputes these enclosures from the recurrence~\eqref{eq:t-rec}; it does not treat the decimal endpoints as input data.

\end{document}